\documentclass{amsart}
\usepackage{amsmath,amssymb,amsfonts}
\usepackage{color}
\usepackage[colorlinks=true,linkcolor=blue,urlcolor=blue,citecolor   = red]{hyperref}

\newtheorem{theorem}{Theorem}[section]
\newtheorem{corollary}[theorem]{Corollary}

\newtheorem{proposition}[theorem]{Proposition}
\newtheorem{remark}[theorem]{Remark}

\begin{document}
\title{Weak compactness of almost limited operators}
\author[A. Elbour]{Aziz Elbour}
\address[A. Elbour]{Universit\'{e} Moulay Isma\"{\i}l, Facult\'{e} des
Sciences et Techniques, D\'{e}partement de Math\'{e}\-matiques, B.P. 509,
Errachidia, Morocco.}
\email{azizelbour@hotmail.com}
\author[N. Machrafi]{Nabil Machrafi}
\address[N. Machrafi]{Universit\'{e} Ibn Tofail, Facult\'{e} des Sciences, D%
\'{e}partement de Math\'{e}\-matiques, B.P. 133, K\'{e}nitra, Morocco.}
\email{nmachrafi@gmail.com}
\author[M. Moussa]{Mohammed Moussa}
\address[M. Moussa]{Universit\'{e} Ibn Tofail, Facult\'{e} des Sciences, D%
\'{e}partement de Math\'{e}\-matiques, B.P. 133, K\'{e}nitra, Morocco.}
\email{mohammed.moussa09@gmail.com}

\begin{abstract}
The paper is devoted to the relationship between almost limited operators
and weakly compacts operators. We show that if $F$ is a $\sigma $-Dedekind
complete Banach lattice then, every almost limited operator $T:E\rightarrow
F $ is weakly compact if and only if $E$ is reflexive or the norm of $F$ is
order continuous. Also, we show that if $E$ is a $\sigma $-Dedekind complete
Banach lattice then the square of every positive almost limited operator $%
T:E\rightarrow E$ is weakly compact if and only if the norm of $E$ is order
continuous.
\end{abstract}

\maketitle

\section{Introduction}

Throughout this paper $X,$ $Y$ will denote real Banach spaces, and $E,\,F$
will denote real Banach lattices. $B_{X}$ is the closed unit ball of $X$ and 
$B_{E}^{+}:=B_{E}\cap E^{+}$ is the positive part of $B_{E}$\textbf{.} We
will use the term operator $T:X\rightarrow Y$ between two Banach spaces to
mean a bounded linear mapping. We refer to \cite{AB3, MN} for unexplained
terminology of the Banach lattice theory and positive operators.

Let us recall that a norm bounded set $A$ in a Banach space $X$ is called 
\emph{limited}, if every weak$^{\ast }$ null sequence $\left( f_{n}\right) $
in $X^{\ast }$ converges uniformly to zero on $A$, that is, $\sup_{x\in
A}\left\vert f_{n}\left( x\right) \right\vert \rightarrow 0$. An operator $%
T:X\rightarrow Y$ is said to be \emph{limited} whenever $T\left(
B_{X}\right) $ is a limited set in $Y$, equivalently, whenever $\left\Vert
T^{\ast }\left( f_{n}\right) \right\Vert \rightarrow 0$ for every weak$%
^{\ast }$ null sequence $\left( f_{n}\right) \subset Y^{\ast }$.

Recently, the authors of \cite{chen} considered the disjoint version of
limited sets by introducing the class of almost limited sets in Banach
lattices. From \cite{chen} a norm bounded subset $A$ of a Banach lattice $E$
is said to be \emph{almost limited}, if every disjoint weak$^{\ast }$ null
sequence $(f_{n})$ in $E^{\ast }$ converges uniformly to zero on $A$.

From \cite{Mach}, an operator $T:X\rightarrow E$ is called \emph{almost
limited} if $T\left( B_{X}\right) $ is an almost limited set in $E$,
equivalently, $\left\Vert T^{\ast }\left( f_{n}\right) \right\Vert
\rightarrow 0$ for every disjoint weak$^{\ast }$ null sequence $\left(
f_{n}\right) \subset E^{\ast }$. Note that an almost limited operator need
not be weakly compact. In fact, the identity operator of the Banach lattice $%
\ell ^{\infty }$ is almost limited but it is not weakly compact.

In this paper, we characterize pairs of Banach lattices $E$, $F$ for which
every almost limited operator $T:E\rightarrow F$ is weakly compact. More
precisely, we will prove that if $F$ is a $\sigma $-Dedekind complete Banach
lattice then, every almost limited operator $T:E\rightarrow F$ is weakly
compact if and only if $E$ is reflexive or the norm of $F$ is order
continuous (\autoref{T2 ess}). Next, we will prove that if $E$ is a $\sigma $%
-Dedekind complete Banach lattice then the square of every positive almost
limited operator $T:E\rightarrow E$ is weakly compact if and only if the
norm of $E$ is order continuous (\autoref{T3}). As consequences, we will
give some interesting results.

\section{Main results}

Let us recall that a Banach lattice $E$ is said to have the dual positive
Schur property if $\left\Vert f_{n}\right\Vert \rightarrow 0$ for every weak$%
^{\ast }$ null sequence $\left( f_{n}\right) \subset \left( E^{\ast }\right)
^{+}$, equivalently, $\left\Vert f_{n}\right\Vert \rightarrow 0$ for every
weak$^{\ast }$ null sequence $\left( f_{n}\right) \subset \left( E^{\ast
}\right) ^{+}$ consisting of pairwise disjoint terms (Proposition 2.3 of 
\cite{WN2012}). A Banach lattice $E$ has the property $(\mathrm{d})$
whenever $\left\vert f_{n}\right\vert \wedge \left\vert f_{m}\right\vert =0$
and $f_{n}\overset{w^{\ast }}{\rightarrow }0$ in $E^{\ast }$ imply $%
\left\vert f_{n}\right\vert \overset{w^{\ast }}{\rightarrow }0$. It should
be noted, by Proposition 1.4 of \cite{WN2012}, that every $\sigma $-Dedekind
complete Banach lattice has the property $(\mathrm{d})$ but the converse is
not true in general. In fact, the Banach lattice $\ell ^{\infty }/c_{0}$\
has the property $(\mathrm{d})$ but it is not $\sigma $-Dedekind complete 
\cite[Remark 1.5]{WN2012}.

Our first result shows that we can restrict sequences appearing in the
definition of almost limited operator $T:X\rightarrow E$ to positive
disjoint sequences if the Banach lattice $E$ has the property $(\mathrm{d})$.

\begin{proposition}
\label{P1}An operator $T:X\rightarrow E$ from a Banach space $X$ into a
Banach lattice $E$ with the property $(\mathrm{d})$, is almost limited if
and only if $\left\Vert T^{\ast }\left( f_{n}\right) \right\Vert \rightarrow
0$ for every weak* null sequence $\left( f_{n}\right) $ in $E^{\ast }$
consisting of positive and pairwise disjoint elements.
\end{proposition}

\begin{proof}
The \textquotedblleft only if\textquotedblright\ part is trivial. For the
\textquotedblleft if\textquotedblright\ part, let $\left( f_{n}\right)
\subset E^{\ast }$ be a disjoint weak* null sequence. As $E$ has the
property $(\mathrm{d})$, $\left\vert f_{n}\right\vert \overset{w^{\ast }}{%
\rightarrow }0$. Using the inequalities $0\leq f_{n}^{+}\leq \left\vert
f_{n}\right\vert $ and $0\leq f_{n}^{-}\leq \left\vert f_{n}\right\vert $,
we see that $\left( f_{n}^{+}\right) $ and $\left( f_{n}^{-}\right) $ are
disjoint weak$^{\star }$ null sequences of $E^{\ast }$. So, from our
hypothesis we see that $\left\Vert T^{\ast }\left( f_{n}^{+}\right)
\right\Vert \rightarrow 0$ and $\left\Vert T^{\ast }\left( f_{n}^{-}\right)
\right\Vert \rightarrow 0$. This implies that $\left\Vert T^{\ast }\left(
f_{n}\right) \right\Vert \rightarrow 0$, and hence $T$ is almost limited.
\end{proof}

The next result follows immediately from Proposition 2.3 of \cite{WN2012}
combined with Proposition \ref{P1}.

\begin{corollary}
A Banach lattice $E$ with the property $(\mathrm{d})$ has the dual positive
Schur property if and only if the identity operator on $E$ is almost limited.
\end{corollary}

The following result shows that if a positive almost limited operator $%
T:E\rightarrow F$ has its range in a Banach lattice with the property $(%
\mathrm{d})$, then every positive operator $S:E\rightarrow F$ that it
dominates (i.e., $0\leq S\leq T$) is also almost limited.

\begin{proposition}
\label{dom}Let $E$ and $F$ be two Banach lattices such that $F$ has the
property $(\mathrm{d})$. If a positive operator $S:E\rightarrow F$ is
dominated by an almost limited operator, then $S$ itself is almost limited.
\end{proposition}

\begin{proof}
Let $S,\quad T:E\rightarrow F$ be two operators such that $0\leq S\leq T$
and $T$ is almost limited. Let $(f_{n})$ be a disjoint sequence in $\left(
F^{\ast }\right) ^{+}$ such that $f_{n}\overset{w^{\ast }}{\rightarrow }0$.
As $T$ is almost limited, $\left\Vert T^{\ast }(f_{n})\right\Vert
\rightarrow 0$. Using the inequalities $0\leq S^{\ast }(f_{n})\leq T^{\ast
}(f_{n})$, we see that $\left\Vert S^{\ast }(f_{n})\right\Vert \leq
\left\Vert T^{\ast }(f_{n})\right\Vert $ for all $n$, from which we get $%
\left\Vert S^{\ast }(f_{n})\right\Vert \rightarrow 0$. Now, by Proposition %
\ref{P1} $S$ is well almost limited.
\end{proof}

The next remark will be useful in further considerations.

\begin{remark}
\label{Rem1}

\begin{enumerate}
\item Consider the scheme of operators $X\overset{R}{\rightarrow }Y\overset{S%
}{\rightarrow }F$. It is easy to see that if $S$ is an almost limited
operator, then $S\circ R$ is likewise almost limited.

\item Consider the scheme of operators $X\overset{R}{\rightarrow }E\overset{S%
}{\rightarrow }F$.

\begin{enumerate}
\item If $R$ is an almost limited operator, then $S\circ R$ is not
necessarily almost limited. In fact, by a result in \cite{Wn1}, there exists
a non regular operator $S:\ell ^{\infty }\rightarrow c_{0}$, which is
certainly not compact. So by Proposition 4.3\ of \cite{Mach}, $S$ is not
almost limited. If $R:\ell ^{\infty }\rightarrow \ell ^{\infty }$ is the
identity operator on $\ell ^{\infty }$ then $R$ is almost limited but $%
S\circ R=S$ is not almost limited.

\item \label{rem}However, if $E$ has the dual positive Schur property (for
example, $E=\ell ^{\infty }$) and $F$ has the property $(\mathrm{d})$, and $S
$ is positive, then $T=S\circ R$ is an almost limited operator. In fact,
according to Proposition \ref{P1}, let $\left( f_{n}\right) \subset F^{\ast }
$ be a positive disjoint weak$^{\ast }$ null sequence. Clearly $0\leq
S^{\ast }f_{n}\overset{w^{\ast }}{\rightarrow }0$ holds in $E^{\ast }$.
Since $E$ has the dual positive Schur property then $\left\Vert S^{\ast
}f_{n}\right\Vert \rightarrow 0$, and hence $\left\Vert T^{\ast
}f_{n}\right\Vert =\left\Vert R^{\ast }\left( S^{\ast }f_{n}\right)
\right\Vert \rightarrow 0$, as desired.
\end{enumerate}
\end{enumerate}
\end{remark}

Our next major result characterizes pairs of Banach lattices $E$, $F$ for
which every positive almost limited operator $T:E\rightarrow F$ is weakly
compact.

\begin{theorem}
\label{T2 ess}Let $E$ and $F$ be two Banach lattices such that $F$ is $%
\sigma $-Dedekind complete. Then the following assertions are equivalent:
\end{theorem}

\begin{enumerate}
\item Every almost limited operator $T:E\rightarrow F$ is weakly compact.

\item Every positive almost limited operator $T:E\rightarrow F$ is weakly
compact.

\item One of the following statements is valid:

\begin{enumerate}
\item $E$ is reflexive.

\item The norm of $F$ is order continuous.
\end{enumerate}
\end{enumerate}

\begin{proof}
$\left( 1\right) \Rightarrow \left( 2\right) $ Obvious.

$\left( 2\right) \Rightarrow \left( 3\right) $ Assume by way of
contradiction that $E$ is not reflexive and the norm of $F$ is not order
continuous. We have to construct a positive almost limited operator $%
T:E\rightarrow F$ which is not weakly compact.

Indeed, since the norm of $F$ is not order continuous, then by Corollary
2.4.3 of \cite{MN} we may assume that $\ell ^{\infty }$ is a closed
sublattice of $F$. As $E$ is not reflexive then $E^{\ast }$ is not
reflexive, and hence the closed unit ball $B_{E^{\ast }}$ of $E^{\ast }$ is
not weakly compact. So, from $B_{E^{\ast }}\subset B_{E^{\ast
}}^{+}-B_{E^{\ast }}^{+}$, we see that $B_{E^{\ast }}^{+}$ is not weakly
compact. Then, by the Eberlein-\v{S}mulian theorem one can find a sequence $%
(f_{n})$ in $B_{E^{\ast }}^{+}$ which does not have any weakly convergent
subsequence. Consider the positive operator $T:E\rightarrow \ell ^{\infty
}\subset F$ defined by%
\begin{equation*}
T\left( x\right) =\left( f_{n}\left( x\right) \right) _{n=1}^{\infty }
\end{equation*}%
for all $x\in E$. By Remark 2.4(\ref{rem}) $T$ is an almost limited operator.
But $T$ is not weakly compact. In fact, if $T$ were weakly compact then $%
T^{\ast }:\left( \ell ^{\infty }\right) ^{\ast }\rightarrow E^{\ast }$ would
weakly compact. Note that $T^{\ast }\left( \left( \lambda _{n}\right)
_{n=1}^{\infty }\right) =\sum\limits_{n=1}^{\infty }\lambda _{n}f_{n}$ for
every $\left( \lambda _{n}\right) _{n=1}^{\infty }\in \ell ^{1}\subset
\left( \ell ^{\infty }\right) ^{\ast }$. So, if $e_{n}$ is the usual basis
element in $\ell ^{1}$ then $T^{\ast }\left( e_{n}\right) =f_{n}$ so that $%
(f_{n})$ would have a weakly convergent subsequence. This contradicts the
choice of $(f_{n})$. Therefore, $T$ is not weakly compact, as desired.

$\left( a\right) \Rightarrow \left( 1\right) $ In this case, every operator
from $E$ into $F$ is weakly compact.

$\left( b\right) \Rightarrow \left( 1\right) $ By Theorem 4.2 of \cite{Mach}
we see that $T$ is L-weakly compact, and by Theorem 5.61 of \cite{AB3} $T$
is well weakly compact.
\end{proof}

By a similar proof as the previous theorem, we obtain the following result.

\begin{theorem}
\label{T2'}Let $X$ a Banach space and $F$ a $\sigma $-Dedekind complete
Banach lattice. Then the following assertions are equivalent:
\end{theorem}

\begin{enumerate}
\item Every almost limited operator $T:X\rightarrow F$ is weakly compact.

\item One of the following statements is valid:

\begin{enumerate}
\item $X$ is reflexive.

\item The norm of $F$ is order continuous.
\end{enumerate}
\end{enumerate}

As a consequence of Theorem \ref{T2 ess}, we obtain an operator
characterization of order continuity of the norm of a $\sigma $-Dedekind
complete Banach lattice.

\begin{corollary}
\label{C1}Let $E$ be a $\sigma $-Dedekind complete Banach lattice. Then the
following statements are equivalent:

\begin{enumerate}
\item Every almost limited operator $T$ from $E$ into $E$ is weakly compact.

\item Every positive almost limited operator $T$ from $E$ into $E$ is weakly
compact.

\item The norm of $E$ is order continuous.
\end{enumerate}
\end{corollary}

Another consequence of Theorem \ref{T2 ess} is the following result.

\begin{corollary}
\label{C2}For a Banach lattice $E$, the following statements are equivalent:

\begin{enumerate}
\item Every positive operator $T:E\rightarrow F$ from $E$ to an arbitrary
infinite dimensional AM-space is weakly compact.

\item Every positive operator $T:E\rightarrow \ell ^{\infty }$ is weakly
compact.

\item $E$ is reflexive.
\end{enumerate}
\end{corollary}

\begin{proof}
$\left( 1\right) \Rightarrow \left( 2\right) $ and $\left( 3\right)
\Rightarrow \left( 1\right) $ are obvious.

$\left( 2\right) \Rightarrow \left( 3\right) $ Follows from Theorem \ref{T2
ess}.
\end{proof}

The following result characterize Banach lattice $E$ for which every
positive almost limited operator $T:E\rightarrow E$ has a weakly compact
square.

\begin{theorem}
\label{T3}Let $E$ be a $\sigma $-Dedekind complete Banach lattice. Then the
following statements are equivalent:

\begin{enumerate}
\item Every positive almost limited operator $T$ from $E$ into $E$ is weakly
compact.

\item For every positive almost limited operator $T$ from $E$ into $E$, the
operator $T^{2}$ is weakly compact.

\item The norm of $E$ is order continuous.
\end{enumerate}
\end{theorem}

\begin{proof}
$\left( 1\right) \Rightarrow \left( 2\right) $ Obvious.

$\left( 2\right) \Rightarrow \left( 3\right) $ Assume by way of
contradiction that the norm of $E$ is not order continuous. So, by Theorem
4.14 of \cite{AB3} there exists a disjoint sequence $(u_{n})\subset E^{+}$
satisfying $\left\Vert u_{n}\right\Vert =1$ and $0\leq u_{n}\leq u$ for all $%
n$ and for some $u\in E^{+}$. We can now proceed analogously to the proof of
Proposition 0.5.5 of \cite{WN}. Let $g_{n}\in E_{+}^{\ast }$ be of norm one
and such that $g_{n}(u_{n})=\left\Vert u_{n}\right\Vert =1$ and let $P_{n}$\
be the band projection onto $\left\{ u_{n}\right\} ^{dd}$, where $\left\{
u_{n}\right\} ^{dd}$ is the band generated by $\left\{ u_{n}\right\} $. If $%
f_{n}=g_{n}\circ P_{n}$, then $f_{n}\wedge f_{m}=0$ for $n\neq m$, $%
\sup_{n}\left\Vert f_{n}\right\Vert \leq 1$ and $f_{n}(u_{m})=\delta _{nm}$.
Hence the operator $S:\ell ^{\infty }\rightarrow E$ defined by%
\begin{equation*}
S(\left( t_{n}\right) _{n=1}^{\infty })=\left( o\right) \sum_{n=1}^{\infty
}t_{n}u_{n}
\end{equation*}%
is a lattice isomorphism from $\ell ^{\infty }$ into $E$, where $\left(
o\right) \sum_{n=1}^{\infty }t_{n}u_{n}$\ denotes the order limit of the
sequence of the partial sums $\sum_{n=1}^{m}t_{n}u_{n}$\ for each $\left(
t_{n}\right) _{n=1}^{\infty }\in \ell ^{\infty }$. Also, let $R:E\rightarrow
\ell ^{\infty }$ be the positive operator defined by%
\begin{equation*}
R(x)=\left( f_{n}(x)\right) _{n=1}^{\infty }.
\end{equation*}%
So, by Remark 2.4(\ref{rem}), the positive operator $T=S\circ R:E\rightarrow F$
defined by%
\begin{equation*}
T(x)=\left( o\right) \sum_{n=1}^{\infty }f_{n}(x)u_{n}
\end{equation*}%
is almost limited. But $T$ is not weakly compact. In fact, let $%
x_{n}=\sum_{k=1}^{n}u_{k}$ for each $n$, and note that $0\leq x_{n}\uparrow
\leq u$. Clearly $T(u_{n})=u_{n}$, and hence $T(x_{n})=x_{n}$\ for all $n$.
If $x$ is a weak limit of a subsequence of $\left( x_{n}\right) $, then it
is easy to see that $x_{n}\uparrow x$ and $x_{n}\overset{w}{\rightarrow }x$
must hold. By Theorem 3.52 of \cite{AB3} we have $\left\Vert
x_{n}-x\right\Vert \rightarrow 0$, and hence $\left\Vert
x_{n+1}-x_{n}\right\Vert \rightarrow 0$. But this contradicts $\left\Vert
x_{n+1}-x_{n}\right\Vert =\left\Vert u_{n+1}\right\Vert =1$ for all $n$.
Thus $\left( x_{n}\right) $ has no weakly convergent subsequence, and hence $%
T$ is not weakly compact, as desired.

$\left( 3\right) \Rightarrow \left( 1\right) $ Follows from Theorem \ref{T2
ess}.
\end{proof}

Finally, note that a weakly compact operator $T:X\rightarrow F$ need not be
almost limited. In fact, the identity operator of the Banach lattice $\ell
^{2}$ is weakly compact but it is not almost limited. However, if $F$ has
the positive Schur property, then the two class coincide. The details follow.

\begin{proposition}
An operator $T:X\rightarrow F$ from a Banach space $X$ to a Banach lattice $F
$ with the positive Schur property is weakly compact if and only if it is
almost limited.
\end{proposition}

\begin{proof}
The \textquotedblleft if\textquotedblright\ part follows from Theorem \ref%
{T2'}. For the \textquotedblleft only if\textquotedblright\ part, assume
that $T:X\rightarrow F$ is weakly compact. It follows from Theorem 3.4 of 
\cite{CW} that $T$ is L-weakly compact, and hence $T$ is almost limited \cite%
[Theorem 4.2]{Mach}.
\end{proof}


\begin{thebibliography}{9}
\bibitem{AB3} Aliprantis, C.D. and Burkinshaw, O.: \emph{Positive operators}%
. Reprint of the 1985 original. Springer, Dordrecht, 2006.

\bibitem{chen} Chen J.X., Chen Z.L., Ji G.X.: \emph{Almost limited sets in
Banach lattices,} J. Math. Anal. Appl. \textbf{412} (2014) 547--553.

\bibitem{CW} Chen, Z.L., Wickstead, A.W.: \emph{L-weakly and M-weakly
compact operators}. Indag. Math. \textbf{10}(3), 321--336 (1999).

\bibitem{Mach} Machrafi N., Elbour A., Moussa M.: \emph{Some
characterizations of almost limited sets and applications}. %
\url{http://arxiv.org/abs/1312.2770}

\bibitem{MN} Meyer-Nieberg P.: \emph{Banach lattices}. Universitext.
Springer-Verlag, Berlin, 1991.

\bibitem{Wn1} Wnuk W.: \emph{A characterization of discrete Banach lattices
with order continuous norms}. Proc. Amer. Math. Soc. \textbf{104}, 197-200
(1988).

\bibitem{WN} Wnuk W.: \emph{Banach lattices with order continuous norms}.
Polish Scientific Publishers PWN, Warsaw (1999).

\bibitem{WN2012} Wnuk W.: \emph{On the dual positive Schur property in
Banach lattices. }Positivity (2013) 17:759--773.
\end{thebibliography}
\end{document}